\documentclass[11pt]{amsart}
\usepackage{geometry}                
\usepackage{graphicx}
\usepackage{amssymb}
\usepackage{epstopdf}
\usepackage[all]{xy}
\newtheorem{theorem}{Theorem}
\newtheorem{lemma}[theorem]{Lemma}
\newtheorem{proposition}[theorem]{Proposition}
\newtheorem{prop}[theorem]{Proposition}

\theoremstyle{definition}
\newtheorem{definition}[theorem]{Definition}
\newtheorem{example}[theorem]{Example}

\newtheorem{corollary}[theorem]{Corollary}

\theoremstyle{remark}
\newtheorem{remark}[theorem]{Remark}

\newcommand{\secat}{{\sf {secat}}}
\newcommand{\tc}{{\sf {TC}}}
\newcommand{\TC}{{\sf {TC}}}
\newcommand{\e}{{\mathfrak {e}}}
\newcommand{\h}{{\mathfrak {h}}}
\newcommand{\Z}{{\mathbb {Z}}}
\newcommand{\Id}{{\sf {Id}}}
\newcommand{\cat}{{\sf {cat}}}
\newcommand{\rk}{{\sf {rk}}}
\renewcommand{\u}{{\mathfrak u}}
\renewcommand{\v}{{\mathfrak v}}
\newcommand{\U}{{\mathfrak U}}
\newcommand{\RP}{{\mathbf {RP}}}
\newcommand{\CP}{{\mathbf {CP}}}
\newcommand{\rr}{{\mathbb {R}}}

\title{Sequential parametrized topological complexity of sphere bundles}
\author{Michael Farber}
	
	\address{School of Mathematical Sciences\\
Queen Mary University of London\\ E1 4NS London\\UK.}
	
	\email{m.farber@qmul.ac.uk}
	
	\address{School of Mathematical Sciences\\
Queen Mary University of London\\ E1 4NS London\\UK.}
	
	\email{amit.paul@qmul.ac.uk/amitkrpaul23@gmail.com}
	
	\subjclass{55M30}
	\keywords{Topological complexity, Parametrized topological complexity, 
	Sequential parametrized topological complexity, Sphere bundle}
	\thanks{Both authors were partially supported by an EPSRC research grant}
\author{Amit Kumar Paul}

\begin{document}
\maketitle

\begin{abstract}
Autonomous motion of a system (robot) is controlled by a motion planning algorithm. 
A sequential parametrized motion planning algorithm \cite{FP22} works under variable external conditions and generates continuous motions of the system to attain the prescribed sequence of states at prescribed moments of time. Topological complexity of such algorithms characterises their structure and discontinuities. 
Information about states of the system consistent with states of the external conditions is described by a fibration $p: E\to B$ where the base $B$ parametrises the external conditions and each fibre $p^{-1}(b)$ is the configuration space of the system constrained by external conditions $b\in B$; more detail on this approach is given below. 
Our main goal in this paper is to study the sequential topological complexity of sphere bundles 
$\dot \xi: \dot E\to B$; in other words we study {\it \lq\lq parametrized families of spheres\rq\rq} and sequential parametrized motion planning algorithms for such bundles. We use the Euler and Stiefel - Whitney characteristic classes to obtain lower bounds on the topological complexity. We illustrate our results by many explicit examples. Some related results for the special case $r=2$ were described earlier in \cite{FW23}.

\end{abstract}
\section{Introduction}

The motion planning problem of robotics  \cite{Lat}, \cite{LaV} leads naturally to 
the notion of topological complexity $\TC(X)$ initiated in \cite{Far03}; $\TC(X)$ is a numerical invariant which characterises behaviour of motion planning algorithms in a variety of different ways, in particular, $\tc(X)$ is the minimal degree of instability of motion planning algorithms for systems having the topological space $X$ as their configuration space, see Theorem 14 in \cite{Far04}. Important further results and generalisations could be found in \cite{BasGRT14}, \cite{Dra}, \cite{Gar19}, \cite {GraLV}. 

A new \emph{\lq\lq parametrized\rq\rq} approach to the motion planning problem was developed recently in \cite{CohFW21}, \cite{CohFW}. Para\-metrized motion planning algorithms can function 
in complex situations involving variable external conditions which are viewed as parameters and constitute part of the input of the algorithm.  
The authors of \cite{CohFW21}, \cite{CohFW} analysed in full detail parametrised topological complexity of the Fadell - Neuwirth bundle \cite{FadN62} which is equivalent to the motion planning problem for many robots and obstacles in the Euclidean space. Further progress was achieved in \cite{FarW} where explicit parametrized motion planning algorithms of minimal complexity were constructed. 

M. Grant \cite{Gra22} applied the notion of parametrized topological complexity to study group epimorphisms. 
In paper \cite{Gra22} he proved a number of interesting bounds and gave a new computation of the parametrised topological complexity of the Fadell-Neuwirth fibration in the planar case, originally computed in \cite{CohFW}. 

In a recent paper \cite{FW23} the authors studied parametrized topological complexity of sphere bundles $\tc[\dot \xi: \dot E\to B]$ 
associated to vector bundles $\xi: E\to B$. It was shown how the properties of the Euler and Stiefel - Whitney characteristic classes 
can help to compute the invariant $\tc[\dot \xi: \dot E\to B]$. The explicitly computed in \cite{FW23} examples show that parametrized topological complexity $\tc[\dot \xi: \dot E\to B]$ can be arbitrarily large, unlike the usual topological complexity of spheres which equals 1 or 2, depending on parity of the dimension. In the other extreme, it was observed in \cite{FW23} that $\tc[\dot \xi: \dot E\to B]=1$ if the vector bundle $\xi$ admits complex structure. 

A generalisation of the concept of parametrised topological complexity was initiated and studied  
in \cite{FP22} where the notions of sequential parametrized motion planning and sequential parametrized topological complexity were developed. The \lq\lq sequential\rq\rq\ approach also allows variable external conditions but the algorithm generates the 
\lq\lq schedule\rq\rq\  for the motion of the system to visit a prescribed sequence of states in a certain order at prescribed moments of time. In the case of fixed external conditions the approach of \cite{FP22}
 reduces to the sequential approach initiated by Y. Rudyak \cite{Rud10}. The main result of \cite{FP22} is the analysis of the sequential parametrised topological complexity of the Fadell - Neuwirth bundle. The case of constant obstacles was analysed earlier by J. Gonzalez and M. Grant \cite{GonG15}. In \cite{FP23} the authors developed explicit sequential parametrized motion planning algorithms for the motion of a number of robots in the presence of an arbitrary number of obstacles
in $\rr^d$. The cases of dimension $d$ being even or odd are essentially distinct. 

In the present article we continue study of sequential parametrized topological complexity with major focus on sphere bundles. The paper consists of the Introduction and sections \S\S \ref{sec:seqpar} -- \ref{sec:9}.

In \S \ref{sec:seqpar} we recall the basic definitions and several general results including the dimensional upper bound and the cohomological lower bound. In \S \ref{sec:sharp} we give a sharp upper bound for the sequential parametrized topological complexity which will be used later in the paper. This result is in the spirit of Theorem 2 and 3 from \cite{FG09} and its proof follows similar lines. 

In \S \ref{sec:upper} we give upper bounds for sequential parametrized topological complexity of sphere bundles. The main result of 
\S \ref{sec:upper}
is Proposition \ref{prop:upper} which involves sectional category of the associated Stiefel bundle $\ddot \xi: \ddot E\to \dot E$. 
In the following \S \ref{sec:5} we consider several Corollaries. 
While for an arbitrary vector bundle 
$\xi$ one has 
$$\tc_r[\dot \xi: \dot E\to B]\ge r-1,$$ we show that for a vector bundle $\xi$ admitting a complex structure one has the equality
\begin{eqnarray}\label{char}
\tc_r[\dot \xi: \dot E\to B]=r-1\quad\mbox{ for every} \quad r=2, 3, \dots.
\end{eqnarray}
This result raises an interesting question of whether equation (\ref{char}) characterises real vector bundles possessing complex structures. 

In section \S \ref{sec:6} we consider the cohomology algebra of the total space of a sphere bundle. 
This problem was studied by W. Massey in \cite{Ma}. 
For our purposes we need only the special case when the bundle admits a continuous section. Our 
Theorem \ref{thm:sec} must be well-known although we do not know a specific reference. 

Theorem \ref{thm:lower} from section \S 7 can be considered to be the main technical result of the paper. It gives an explicit expression for the cup-length of the kernel of the diagonal map which appears in Proposition \ref{lem:lowerbound} and Theorem \ref{thm:sharp} culminating in Corollary \ref{cor:lbound}. In \S \ref{sec:8} we consider several special cases when the upper and lower bounds match and the final answer can be given. We may mention Proposition \ref{thm:31} which allows to simplify the lower bounds in the case when the original bundle $\xi$ has a nonzero section. 

In the final section \S \ref{sec:9} we present an estimate using the Stiefel - Whitney characteristic classes; we also compute a few examples.

\section{Sequential parametrized topological complexity}\label{sec:seqpar}

In this section we review the notion of sequential parametrized topological complexity introduced in \cite{FP22}, see also \cite{CohFW21}, \cite{CohFW}, \cite{FP23}. This notion is motivated by the motion planning problem of robotics. We consider an autonomous robot moving in varying external conditions which are parametrized by a topological space $B$. 
An important special case is when the external conditions are the positions of the obstacles; in this case the space $B$ is the appropriate configuration space. Mathematically, it is convenient to consider a general situation when $B$ is an arbitrary topological space. We formalise the problem by considering a Hurewicz fibration $p : E \to B$ where for a choice of external conditions $b\in B$ the fibre $X_b= p^{-1}(b)$ 
is the space of all configurations of the system consistent with external conditions $b$. The total space $E$ is the union of the fibres 
$E=\sqcup_{b\in B}X_b$ and its topology reflects natural connectivity of the situation. We refer the reader to \cite{FarW}, \cite{FP22}, \cite{FP23} where the motion planning problem of many autonomously moving objects in the presence of multiple obstacles was studied.

Let $p: E\to B$ be a Hurewicz fibration. As usual, the symbol $E^I$ stands for the space of all continuous paths $\alpha: I=[0,1]\to E$. We denote by $E^I_B\subset E^I$  
the space of all paths $\alpha: I\to E$ such that $p\circ \alpha: I\to B$ is constant. 
Such paths $\alpha: I\to E$ represent motions of the system under fixed external conditions. 
Fix $r\ge 2$ points 
$0\le t_1<t_2<\dots <t_r\le 1$ in $I$  (\lq\lq the time schedule\rq\rq)
and consider the evaluation map 
\begin{eqnarray}\label{Pir}
\Pi_r : E^I_B \to E^r_B, \quad \Pi_r(\alpha) = (\alpha(t_1), \alpha(t_2), \dots,  \alpha(t_r)).\end{eqnarray} 
where for an integer $r\ge 2$ we denote 
$$E^r_B= \{(e_1, \cdots, e_r)\in E^r; \, p(e_1)=\cdots = p(e_r)\}.$$ 

A section $s: E^r_B \to E^I$ of the fibration $\Pi_r$ can be interpreted as {\it a parametrized sequential motion planning algorithm}, i.e. it is 
a function which assigns to every sequence of points $(e_1, e_2, \dots, e_r)\in E^r_B$ a continuous path $\alpha: I\to E$ (\lq\lq the motion of the system\rq\rq) satisfying $\alpha(t_i)=e_i$ for every $i=1, 2, \dots, r$ and such that the path 
$p\circ \alpha: I \to B$ is constant. The latter condition means that the system moves under the constant external conditions 
(such as positions of the obstacles). 

$\Pi_r$ is a Hurewicz fibration, see \cite[Appendix]{CohFW}; the  fibre of $\Pi_r$ is $(\Omega X)^{r-1}$ where $X$ is the fibre of $p: E\to B$. 
Typically $\Pi_r$ does not admit continuous sections and hence the 
motion planning algorithms are generally discontinuous. 

The following definition gives a measure of the complexity of sequential parametrized motion planning algorithms.

\begin{definition}\label{def:main}
The {\it $r$-th sequential parametrized topological complexity} of the fibration $p : E \to B$, denoted $\TC_r[p : E \to B]$, is defined as the sectional category of the fibration $\Pi_r$, i.e. 
\begin{eqnarray}\label{tcsec}
\TC_r[p : E \to B]:=\secat(\Pi_r).
\end{eqnarray}
\end{definition}	

In more detail, $\TC_r[p : E \to B]$ is defined as the minimal integer $k$ such that there is a open cover 
$\{U_0, U_1, \dots, U_k\}$ of $E^r_B$ with the property that each open set $U_i$ 
admits a continuous section $s_i : U_i \to E^I_B$ of $\Pi_r$.
The following Lemma allows using arbitrary partitions instead of open covers: 

 \begin{lemma}[see Proposition 3.6 in \cite{ FP22}]\label{lemma para tc}
 Let $p: E \to B$ be a locally trivial fibration where $E$ and $B$ are metrisable separable ANRs. Then the $r$-th sequential parametrized topological complexity $\TC_r[p: E \to B]$ equals the smallest integer $k\ge 0$ such that the space $E_B^r$ admits a partition $$E_B^r=F_0 \sqcup F_1 \sqcup ... \sqcup F_k, \quad F_i\cap F_j= \emptyset \text{ \ for } i\neq j,$$
and on each set $F_i$ there exists a continuous section $s_i : F_i \to E_B^I$ of the fibration $\Pi_r$.
 \end{lemma}

\subsection*{Alternative descriptions of sequential parametrized topological complexity} In certain situations it is convenient  to generalise the above definition of $\TC_r[p: E \to B]$ as follows. 
Let $K$ be a path-connected finite CW-complex and let $k_1, k_2, \cdots, k_r\in K$ be a collection of $r$ pairwise 
distinct points of $K$, where $r\ge 2$. For a Hurewicz fibration $p: E \to B$, consider the space $E^{K}_B$ of all continuous maps 
$\alpha: K \to E$ such that the composition $p\circ \alpha: K\to B$ is a constant map. 
We equip $E^K_B$ with the compact-open topology induced from the function space $E^K$. 
Consider the evaluation map
$$\Pi_r^K : E^{K}_B \to E^r_B, \quad \Pi^K_r(\alpha) = (\alpha(k_1), \alpha(k_2), \cdots, \alpha(k_r)) \quad \mbox{for}
\quad\alpha\in E^K_B.$$ 
$\Pi^K_r$ is a Hurewicz fibration by the result of the Appendix to \cite{CohFW}.

\begin{lemma}\label{lemma para tc by secat} {\rm (See \cite{FP22}, Lemma 3.5)}  For any path-connected finite CW-complex $K$
and a set of pairwise distinct points $k_1, \dots, k_r\in K$ one has $$\secat(\Pi^K_r) = \TC_r[p:E\to B].$$
 \end{lemma}		
\begin{proof} Let $0\le t_1<t_2<\dots<t_r\le 1$ be a given time schedule used in the definition of the map $\Pi_r=\Pi_r^I$ given by (\ref{Pir}). 
Since $K$ is path-connected we may find a continuous map $\gamma: I\to K$ with $\gamma(t_i) =k_i$ for all $i=1, 2, \dots, r$. We obtain a continuous map $F_\gamma: E^K_B \to E^I_B$ acting by the formula $F_\gamma(\alpha) = \alpha \circ \gamma$. It is easy to see that the following diagram commutes
$$
 \xymatrix{
E_B^{K} \ar[rr]^{F}  \ar[dr]_{\Pi_{r}^K}& &E_B^{I} \ar[dl]^{\Pi_{r}^I} \\ & E_B^r
}$$
Thus, any partial section $s: U\to E^K_r$ of $\Pi_r^K$ defines a partial section $F\circ s$ of $\Pi_r^I$ implying
$$\TC_r[p:E\to B]= \secat(\Pi^I_r)\le \secat(\Pi_r^K).$$
To obtain the inverse inequality note that any locally finite CW-complex is metrisable. Applying Tietze extension 
theorem we can find continuous functions $\psi_1, \dots, \psi_r: K\to [0,1]$ such that $\psi_i(k_j)=\delta_{ij}$, 
i.e. $\psi_i(k_j)$ equals 1 for $j=i$ and it equals $0$ for $j\not=i$. The function $f=\min\{1, \sum_{i=1}^r t_i\cdot \psi_i\}: K\to [0,1]$
has the property that $f(k_i)=t_i$ for every $i=1, 2, \dots, r$. We obtain a continuous map $F': E^I_B \to E^K_B$, where $F'(\beta) = \beta\circ f$, \, $\beta\in E^I_B$, which appears in the commutative diagram
$$
 \xymatrix{
E_B^{I} \ar[rr]^{F'}  \ar[dr]_{\Pi_{r}^I}& &E_B^{K} \ar[dl]^{\Pi_{r}^K} \\ & E_B^r
}$$
As above, this implies the opposite inequality 
$\secat(\Pi_r^K) \le \secat(\Pi^I_r)$
and completes the proof. 
\end{proof}
\begin{example}
As an example consider the case when $p: E \to B$ is a principal bundle with a connected topological group $G$ as fibre. Then $$\TC_{r}[p: E \to B] = \cat(G^{r-1})=\TC_r(G),$$
see \cite{FP22}, Proposition 3.3. In particular, for the Hopf bundle $p: S^3\to S^2$ one has $$\tc_r[p: S^3\to S^2]=\cat((S^1)^{r-1}) =r-1.$$
\end{example}

The following result (see Proposition 6.3 from \cite{FP22}) provides a lower bound for $\TC_{r}[p: E \to B]$:

\begin{prop}\label{lemma lower bound for para tc}\label{lem:lowerbound}
For a fibration $p: E \to B$, consider the diagonal map $\Delta : E \to E^r_B$ where $\Delta(e)= (e, e, \cdots, e)$, and the induced by $\Delta$ homomorphism in cohomology $$\Delta^\ast: H^\ast(E^r_B;R) \to H^\ast(E;R)$$ with coefficients in a ring $R$. 
If there exist cohomology classes $$u_1, \cdots, u_k \in \ker[\Delta^* : H^*(E_B^r; R) \to H^*(E; R)]$$ 
such that 
$$u_1 \cup \cdots \cup u_k \neq 0 \in H^*(E_B^r; R),$$ then $\TC_{r}[p: E \to B]\geq k$. 
\end{prop}

The upper bound is given by Proposition 6.1 from \cite{FP22} which states:

\begin{prop}\label{prop upper bound}
	Let $p: E \to B$ be a locally trivial fibration with fiber $X$, where $E, B, X$ are CW-complexes. Assume that the fiber $X$ is $k$-connected, where $k\ge 0$. Then 
	\begin{eqnarray}\label{upper}
	\TC_{r}[p: E \to B]\leq \left\lceil
	\frac{r\cdot \dim X+\dim B -k}{k+1}\right\rceil.
	\end{eqnarray}
\end{prop}

\begin{example}
Let $\xi: E\to B$ be a rank $q\ge 2$ vector bundle over a CW-complex $B$. Consider the unit sphere bundle $\dot \xi: \dot E\to B$. 
Its fibre is the sphere $S^{q-1}$ and the upper bound (\ref{upper}) gives
\begin{eqnarray}\label{upper2}
\tc_r[\dot \xi: \dot E\to B] \le \left\lceil \frac{r(q-1)+\dim B -q+2}{q-1}\right\rceil = r-1 +\left\lceil \frac{\dim B+1}{q-1}\right\rceil.
\end{eqnarray}
\end{example}

In \S \ref{sec:sharp} below we give statements allowing sharper upper bounds in some cases.

We finish this section with a well-known statement which will be used later in this paper.
\begin{lemma}\label{part}
Let $p: E \to B$ be a fibration where the base $B$ is metrizable. Then there exists a partition 
$B=A_0 \sqcup A_1 \sqcup \dots \sqcup A_k$
with $k=\secat[p:E\to B]$
such that each set $A_i$ admits a continuous section of $p$ and, additionally, 
\begin{eqnarray}\label{closure}
\bar{A_i}\subset \bigcup_{j\geq i} \ A_j\end{eqnarray}
for every $i=0, 1, \dots, k$. 
\end{lemma}

\begin{proof}
For $k=\secat[p:E\to B]$, let $U_0 \cup U_1 \cup \dots \cup U_k=B$ be an open cover  
with continuous sections $s_i : U_i \to E$ of $p$ where $i=0, 1, \dots, k$. 

For $b\in B$ we denote by $\mu(b)$ the cardinality of the set 
$ \{i \ ; \ b\in U_i\},$ i.e. $\mu(b)\in \{1, 2, \dots, k+1\}$ is the number of sets $U_i$ containing the point $b$.  
If a sequence of points $b_n\in B$ with $\mu(b_n)=i$ converges to a point $b_0\in B$ then clearly
$\mu(b_0)\le i$. 

Denote $$A_i=\{b\in B \ ; \mu(b)=k+1-i\}, \quad 0\leq i \leq k.$$ Clearly the sets $\{A_i\}$ are pairwise disjoint and $B=A_0 \sqcup A_1 \sqcup \dots \sqcup A_k.$
To prove (\ref{closure}) assume that $b_0\in \bar{A_i}$. This means that $b_0=\lim b_n$ where $b_n \in A_i$, i.e. $\mu(b_n)=k+1-i$. By the remark above $\mu(b_0) \le k+1-i$, i.e. $b_0\in A_j$ where $j\ge i$. 

For a sequence $\alpha =(i_1<i_2< \dots <i_p)$, where $0\le i_1< i_2< \dots < i_p\le k$, we denote 
$U_\alpha= U_{i_1}\cap U_2\cap \dots\cap U_{i_p}$ and $|\alpha|=p$. For two such sequences we shall write 
$\alpha>\beta$ if $\beta$ is a subsequence of $\alpha$.  We shall also introduce the notation 
$$U'_\alpha = U_\alpha - \bigcup_{|\beta|>|\alpha|} U_\beta.$$
The set $A_i$ can be represented as the disjoint union
$$A_i =\bigcup_{|\alpha|= k+1-i} U'_\alpha.$$
Similar to the argument given above, one has
$\overline {U'_\alpha} \subset \cup_{\beta<\alpha}U'_\beta$
and therefore 
$\overline{U'_\alpha} \cap U'_{\alpha'}=\emptyset$ for every pair $\alpha\not=\alpha'$ satisfying 
$|\alpha|=|\alpha'|$. Thus, the section $s_{i_1}$ is defined and is continuous on $U'_\alpha$ where
$\alpha =(i_1<i_2< \dots <i_p)$. Moreover, the sections $s_0, s_1, \dots, s_k$  jointly define as explained above continuous sections over the sets $A_i$. 
\end{proof}
\section{Sharp upper bound for $\tc_r[p:E\to B]$.}\label{sec:sharp}

The results of this section (Theorem \ref{sharp} and Corollary \ref{sharp1}) sometimes allow to sharpen the upper bound given by Proposition \ref{prop upper bound}; this will be illustrated by examples in \S \ref{sec:8}.


\begin{theorem}\label{sharp}\label{thm:sharp}
Let $p: E\to B$ be a locally trivial bundle with fibre $X$ where the spaces $X, E, B$ are finite CW-complexes. Assume that $X$ is 
$k$-connected, where $k\ge 1$, and  the fraction
$$\frac{\dim E^r_B}{k+1} = \frac{r\cdot \dim X + \dim B}{k+1}=m$$ is an integer, where $r\ge 2$.  
Then: 

(A) One has $\tc_r[p:E\to B] \le m$
and the equality 
\begin{eqnarray}\label{eq}
\tc_r[p:E\to B] = m
\end{eqnarray} holds if and only if there exists 
a cohomology class
$\theta\in H^{k+1}(E^r_B;\mathcal A),$ 
where $\mathcal A$ is a local coefficient system over the space $E^r_B$, 
such that $\Delta^\ast \theta =0\in H^{k+1}(E; \mathcal A|_E)$ and the $m$-th power
$$
\theta^{m}=\theta\cup \theta \cup \dots\cup \theta \not=0 \in H^{(k+1)m}(E^r_B; \mathcal A^{\otimes m})
$$
is nonzero. Here $\Delta: E\to E^r_B$ denotes the diagonal inclusion. 

(B) If, additionally to the above assumptions, the base $B$ is simply connected and the homology group $H_{k+1}(X;\Z)$ is 
torsion free then  
the equality (\ref{eq}) holds if and only if there exist integral cohomology classes 
$v_1, v_2, \dots, v_{m}\in H^{k+1}(E^r_B;\Z)$ such that 
$\Delta^\ast v_i=0$ for  $i=1, 2, \dots, m$
and the cup-product
$$v_1v_2\dots v_{m} \not=0\in H^{(k+1)m}(E^r_B;\Z)$$
is nonzero.  
\end{theorem}
\begin{proof} To prove statement (A) we note that the inequality $\tc_r[p:E\to B] \le m$ follows from Proposition \ref{prop upper bound}. The existence of a class $\theta\in H^{k+1}(E^r_B;\mathcal A)$ satisfying $\Delta^\ast(\theta)=0$ and $\theta^m\not=0$ implies $\tc_r[p:E\to B] \ge m$, by Proposition \ref{lem:lowerbound}. We only need to prove the \lq\lq only if\rq\rq\  part of the statement. 

Indeed, assuming that $\tc_r[p:E\to B] = m$ we may apply Theorem 3 of A. Schwarz \cite{Sva66} which implies that the 
 fibrewise join $\Pi_r^{\ast m}$
 of $m$ copies of the fibration (\ref{Pir}) admits no continuous sections.
 This fibrewise join $\Pi_r^{\ast m}$ has as its fibre the $m$-fold join 
 \begin{eqnarray}\label{fib}
 F_{m, r}= F_r\ast F_r\ast \dots\ast F_r
  \end{eqnarray}
  where $F_r=(\Omega X)^{r-1}$ 
is the fibre of (\ref{Pir}). 
 
 Since $X$ is $k$-connected, the loop space $\Omega X$ is $(k-1)$-connected and hence the fibre $F_{m, r}$ is $c$-connected where
 $$c=m(k-1) + 2(m-1)= m(k+1) -2 = \dim E^r_B -2.$$  
 Here we use the well-known fact that the join of a $p$-connected space and a $q$-connected space is $(p + q + 2)$-connected.

The first obstruction $\Theta$ for a section of $\Pi_r^{\ast m}$ lies in the cohomology group 
$$
\Theta \in H^{m(k+1)}(E^r_B; \{\pi_{c+1}(F_{m, r})\}).
$$
Here $ \{\pi_{c+1}(F_{m, r})\}$ is the local system on $E^r_B$ formed by the homotopy groups of the fibres. By 
Hurewicz theorem one has an isomorphism of local systems $$ \{\pi_{c+1}(F_{m, r})\}\simeq {\mathcal H}_{c+1}(F_{m, r}).$$ 
Here ${\mathcal H}_{c+1}(F_{m, r})$ is the local system of homology groups $H_{c+1}(F_{m, r})$ of the fibres of the fibration 
$\Pi_r^{\ast m}$. 

Note that $\Theta\not=0$ is nonzero since otherwise the fibration $\Pi_r^{\ast m}$ would admit a continuous section as all higher order obstruction clearly vanish by dimensional reasons.

Using the K\"unneth formula (see Chapter 1, \cite{Sva66}) we can write 
an isomorphism of local systems 
$$
{\mathcal H}_{c+1}(F_{m, r})\simeq {\mathcal H}_k(F_r)\otimes {\mathcal H}_k(F_r) \otimes \dots\otimes {\mathcal H}_k(F_r)
$$
where $\mathcal A = {\mathcal H}_k(F_r)$ 
is the local system of the homology groups of the fibres of the original fibration (\ref{Pir}).  
By Theorem 1 from (\cite{Sva66}) one has 
$$\Theta = \theta \cup \theta\cup \dots \cup \theta= \theta^m\not = 0$$
where $$\theta\in H^{k+1}(E^r_B; {\mathcal H}_k(F_r))$$ is the first obstruction for a section of the fibration (\ref{Pir}). 

Note that $\Delta^\ast(\theta)=0$ as the fibration (\ref{Pir}) has an obvious section over the diagonal (the constant paths). This completes the proof of (A).

In the case (B) the base $B$ is simply connected and hence the local coefficient system 
$\mathcal A = {\mathcal H}_k(F_r)$ is trivial and the abelian group $H_k(F_r)$ is free abelian, i.e. 
$H_k(F_r)=\oplus_{i\in I} A_i$ where each $A_i\simeq \Z$.  Therefore the tensor power $H_k(F_r)^{\otimes m} $ is the direct sum 
$$H_k(F_r)^{\otimes m} = \bigoplus_{i_1, \dots, i_m\in I} A_{i_1}\otimes \dots \otimes A_{i_m}$$
where each tensor product $A_{i_1}\otimes \dots \otimes A_{i_m}\simeq \Z$.
Since the class $$\Theta = \theta^m\in H^{m(k+1)}(E^r_B; H_k(F_r)^{\otimes m})$$ 
is nonzero at least one of its images under the coefficient homomorphisms 
$H_k(F_r)^{\otimes m} \to A_{i_1}\otimes \dots \otimes A_{i_m}$
is nonzero. In that case, the product of the classes 
$$v_{i_s} \, \in\, H^{k+1}(E^r_B; A_{i_s}), \quad s=1, 2, \dots, m$$
is nonzero and clearly $\Delta^\ast(v_{i_s})=0.$
\end{proof}

\begin{corollary}\label{sharp1} Let $p: E\to B$ be a locally trivial bundle with fibre $X$ where the spaces $X, E, B$ are finite CW-complexes. 
Assume that $X$ is $k$-connected, where $k\ge 1$, and  $H_{k+1}(X;\Z)$ is torsion free. Additionally, assume that the base $B$ is simply connected and $$r\cdot \dim X + \dim B=(k+1)m$$ for some integers $r\ge 2$ and $m\ge 2$.  If the cup-length of the kernel
\begin{eqnarray}\label{ker}
\ker[\Delta^\ast: H^{\ast}(E^r_B;\Z)\to H^{\ast}(E;\Z)]
\end{eqnarray}
is less than $m$ then $\tc_r[p:E\to B]< m$. 
\end{corollary}

\begin{proof}
This Corollary is a reformulation of the part (B) of Theorem \ref{sharp}. We only need to add that the kernel (\ref{ker}) contains no classes of degree $\le k$. This can be easily seen from the spectral sequence of the fibration 
$$
\pi_r: E^r_B\to E, \quad \pi_r(e_1, \dots, e_r)=e_1,
$$
for which the diagonal map $\Delta: E\to E^r_B$ is a section. 
\end{proof}

The results of this section will be applied later in this paper for computing the sequential parametrized topological complexity of sphere bundles.

\section{Upper bounds for the sequential parametrized topological complexity of sphere bundles}\label{sec:upper}
Let $\xi: E\to B$ be a locally trivial vector bundle of rank $q=\rk(\xi)$ equipped with a metric structure, i.e. with a continuous scalar product of each fibre. We shall denote $E=E(\xi)$ when dealing with several bundles at once. 

Let $\dot \xi: \dot E\to B$ denote the unit sphere bundle of $\xi$; the fibre of $\dot \xi$ is the sphere $S^{q-1}$. 
Our goal in this paper is to estimate the sequential parametrized topological complexity
${\sf {TC}}_r[\dot\xi: \dot E\to B].$

To state our first result we need to introduce some notations. We shall denote by $\ddot E$ the space of pairs $(e, e')$ of unit vectors $e, e'\in \dot E$ such that $\xi(e)=\xi(e')$ and $e\perp e'$. In other words, $\ddot E$ is the total space of the bundle of Stiefel manifolds associated with $\xi$.  Let $\ddot \xi: \ddot E\to \dot E$ denote the projection $\ddot \xi(e, e')=e$. Clearly 
$\ddot \xi: \ddot E\to \dot E$ is a bundle with fibre $S^{q-2}$.

\begin{proposition} \label{prop:upper} One has
$${\sf {TC}}_r[\dot\xi: \dot E\to B] \le  {\sf {secat}}[\ddot\xi: \ddot E\to \dot E] + r-1.$$
\end{proposition}

\begin{proof}
Let $Y=Y_r$ denote the graph with a central vertex $k_1$ and the vertexes $k_2, \dots, k_r$, each connected to $k_1$ by an edge as shown on Figure \ref{graph}.
\begin{figure}[h]
\begin{center}
\includegraphics[scale=0.5]{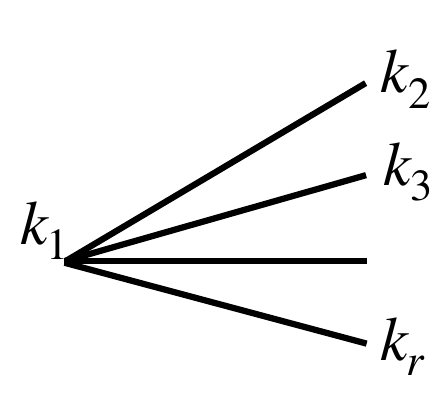}
\caption{Graph $Y_r$}
\label{graph}
\end{center}
\end{figure}
Consider the map 
\begin{eqnarray}\label{Pi}
\Pi^Y_r: \dot E_B^Y\to \dot E_B^r
\end{eqnarray}
which acts as the evaluation of a map $\psi: Y\to X$ at the points $k_1, \dots, k_r$, i.e. $\Pi^Y_r(\psi) = (\psi(k_1), \psi(k_2), \dots, \psi(k_r)).$ By Lemma \ref{lemma para tc by secat}, the sectional category of this map equals $\tc_r[\dot \xi:\dot E\to B]$. 
A partial section $s$ of $\Pi$ maps a configuration of points $(e_1, \dots, e_r)$ lying in a single fibre (i.e. such that $\dot \xi(e_1) =\dot\xi(e_2)=\dots=\dot\xi(e_r))$ to a map $\psi: Y\to \dot E$ with 
$\dot\xi\circ \psi={\rm const}$ and 
$\Pi^Y_r(\psi)= (e_1, \dots, e_r)$, i.e. $\psi(k_i)=e_i$ for $i=1, \dots, r$. In other words, to specify a section we need to specify $r-1$ paths 
$$\gamma_2, \gamma_3, \dots, \gamma_r\in (\dot E)_B^I$$ leading from 
$e_1$ to each of the points $e_2, \dots, e_r$ correspondingly, i.e. $\gamma_i(0)=e_1$ and $\gamma_i(1)=e_i$ for $i=2, 3, \dots, r$.

\subsection{Sets $A_i$} Applying Lemma \ref{part} to the bundle $\ddot \xi: \ddot E\to \dot E$ we obtain a partition 
\begin{eqnarray}
\dot E= A_0\sqcup A_1\sqcup \dots\sqcup A_k, \quad \mbox{where}\quad k=\secat[\ddot \xi: \ddot E\to \dot E],
\end{eqnarray}
such that $\overline A_i\subset \cup_{j\ge i}A_j$ and there exist continuous sections $s_i: A_i\to \ddot E$ of the fibration $\ddot \xi$ for every $i=0, 1, \dots, k$.
\subsection{The set $F_0$} 
%
%
%
%
The set $F_0\subset \dot E^r_B$ is defined as follows:
$$F_0=\{(e_1, \dots, e_r)\in \dot E^r_B; e_2\not=-e_1, \, e_3\not= - e_1, \dots, e_r\not=-e_1\}.$$
We can specify a continuous section of $\Pi^Y_r$ over $F_0$ by defining the paths $\gamma_2, \dots, \gamma_r$, where
\begin{eqnarray}\label{notinj}
\gamma_j(t) = \frac{(1-t)e_1+ te_j}{||(1-t)e_1+ te_j||}, \quad j=2, 3, \dots, r\end{eqnarray}
as explained above. 

\subsection{The sets $F^i_{J}$} 
For a nonempty subset $J\subset \{2, 3, \dots, r\}$ and for an index $i=0, 1, \dots k$, we denote by 
$F^i_{J}\subset \dot E^r_B$ the set of all configurations $(e_1, \dots, e_r)\in \dot E^r_B$ such that 
\begin{enumerate}
\item[{(1)}] $e_1\in A_i$; 
\item[{(2)}] $e_j=-e_1$ for all $j\in J$; 
\item[{(3)}] $e_j\not=-e_1$ for all $j\notin J$. 
\end{enumerate}
Each set $F_J^i$ admits a continuous section of the fibration (\ref{Pi}). Indeed, for $j\in J$ we can define 
$$
\gamma_j(t) = \cos(\pi t)\cdot e_1 + \sin(\pi t)\cdot s_i(e_1), \quad t\in [0,1].
$$
For $j\notin J$ we define $\gamma_j(t)$ by formula (\ref{notinj}). 

The closure of the set $F_J^i$ satisfies
\begin{eqnarray}
\overline{F^i_J} \subset \bigcup_{J'\supset J, \, i'\ge i}F^{i'}_{J'}.
\end{eqnarray}
In particular we see that 
$$
\overline{F^i_J}\cap F^{i'}_{J'}=\emptyset\quad \mbox{if}\quad |J|+i=|J'|+i'.
$$
\subsection{The sets $G_s$} Define 
$$G_s= \bigsqcup_{|J|+i=s} F^i_J, \quad s=1, 2, \dots, k+r-1.$$
The sections described above define jointly a continuous section of (\ref{Pi}) over each $G_s$. Together with
the set $F_0=G_0$ defined earlier we obtain a partition
$$G_0\sqcup G_1\sqcup \dots\sqcup G_{k+r-1}$$ implying that 
$\tc_r[\ddot \xi: \ddot E\to \dot E] \le k+r-1.$
This completes the proof. 
\end{proof}

\section{Corollaries of Proposition \ref{prop:upper}}\label{sec:5}
\begin{corollary}\label{cor:complex}
Let $\xi: E\to B$ be a vector bundle admitting a complex structure. Then 
\begin{eqnarray}\label{r-1}
\tc_r[\dot \xi: \dot E\to B]= r-1\quad\mbox{for all}\quad r=2, 3, \dots
\end{eqnarray}
\end{corollary}
\begin{proof}
If the bundle $\xi$ admits a complex structure then $$\secat[\ddot \xi: \ddot E\to \dot E]=0.$$ Indeed, the multiplication by $\sqrt{-1}$ defines a continuous section of the bundle $\ddot \xi: \ddot E\to \dot E$ which can be explicitly given by 
$s(e) = (e, e')$ where $e'=\sqrt{-1}\cdot e$. Applying Proposition \ref{prop:upper} gives $\tc_r[\dot \xi: \dot E\to B]\le  r-1$. 
On the other hand, $$\tc_r[\dot \xi: \dot E\to B]\ge \tc_r(S^{q-1}) = r-1,$$ where $q=\rk(\xi)$. Here we first used the well-known inequality relating the sequential parametrized topological complexity with the sequential topological complexity of the fibre; on the second step we used the known fact that the sequential topological complexity of an odd dimensional sphere is $r-1$. 
\end{proof}

\begin{remark}
Clearly the value $r-1$ is the lower bound for $\tc_r[\dot \xi: \dot E\to B]$ for any vector bundle $\xi$. Moreover, $\tc_r[\dot \xi: \dot E\to B]\ge r$ if the rank $q=\rk (\xi)$ is odd. One may ask if there exist vector bundles $\xi$ admitting no complex structure and such that 
$\tc_r[\dot \xi: \dot E\to B]=r-1$ for all $r$? In other words we ask whether equation (\ref{r-1}) characterises the class of vector bundles admitting a complex structure. 
\end{remark}

Next we mention the following inequality which can be compared with the inequality of Lemma 5.9 from \cite{FW23}.

\begin{proposition}\label{prop:upper1} Let the vector bundle $\xi: E(\xi)\to B$ be the Whitney sum $\xi=\xi_1\oplus \xi_2$. 
Then 
\begin{equation*}
\secat[\ddot\xi: \ddot E(\xi)\to \dot E(\xi)] \le \secat[\dot \xi_1: \dot E(\xi_1)\to B] +\secat[\dot \xi_2: \dot E(\xi_2)\to B] +1. 
\end{equation*}
 \end{proposition}

\begin{proof} Let us assume that $\xi=\xi_1 \oplus \xi_2$. We want to estimate above the sectional category 
$\secat[\ddot\xi: \ddot E(\xi)\to \dot E(\xi)]$. A vector $e\in \dot E(\xi)$ can be represented as a sum
 $$ e= e_1 +e_2, \quad |e_1|^2+|e_2|^2 =1,$$ 
where $e_1\in E(\xi_1)$ and $e_2\in E(\xi_2)$.

Denote by $G_0\subset \dot E(\xi)$ the set of unit vectors $e= e_1 +e_2 $ such that $e_1\not=0$ and $e_2\not=0$. We define the section $s_0: G_0\to \ddot E(\xi)$ by 
$s_0(e) =((e_1, e_2), (e'_1, e'_2))\in \ddot E(\xi)$ where 
$$e'_1 =- \frac{e_1}{|e_1|}\cdot |e_2|, \quad e'_2= \frac{e_2}{|e_2|}\cdot |e_1|.$$

For $\alpha=1, 2$, denote $k_\alpha=\secat[\dot\xi_\alpha: \dot E(\xi_\alpha)\to B]$ and let 
$B=U^\alpha_0\cup U_1^\alpha\cup \dots\cup U_{k_\alpha}^\alpha$ be an open cover with each set $U_j^\alpha$ admitting a continuous section of the bundle $\dot \xi_\alpha: \dot E(\xi_\alpha)\to B$. Using Lemma 3.5 from \cite{FGLO} we may find an open cover
$$B = W_0\cup W_1\cup \dots W_k, \quad\mbox{where}\quad k=k_1+k_2,$$
such that each open set $W_j$ admits a continuous section $\sigma_j^\alpha: W_j\to \dot E(\xi_\alpha)$ 
for both bundles $\alpha=1$ and $\alpha=2$. 
%

For $j=0, 1, \dots, k$ define the set $F_j=F_j^1\sqcup F_j^2\subset \dot E(\xi)$ as follows 
$$F_j^\alpha=\xi_\alpha^{-1}(W_j) \, \subset \dot E(\xi_\alpha)\subset \dot E(\xi), \quad \alpha= 1, 2. $$
We may define a continuous section $s_j^\alpha: F_j^\alpha \to \ddot E(\xi)$ by the formula
\begin{eqnarray}
s_j^\alpha(e) = (e, \sigma_j^\beta(\xi_\alpha(e))), \quad \beta\not=\alpha, \quad \beta\in \{1, 2\}.
\end{eqnarray}
Two sections $s_j^1$ and $s_j^2$ jointly define a continuous section $s_j: F_j\to \ddot E(\xi)$ since obviously
$\overline {F_j^1}\cap \overline {F_j^2}=\emptyset$.

Now we see that the sets 
$$G_0, F_0, F_1, \dots, F_k$$
cover $\dot E(\xi)$ and each of these sets admits a continuous section of $\ddot\xi: \ddot E(\xi)\to \dot E(\xi)$. 
This gives the inequality 
$$\secat[\ddot\xi: \ddot E(\xi)\to \dot E(\xi)] \le k+1.$$
This completes the proof.
\end{proof}

\begin{corollary}
If $\xi=\eta\oplus \epsilon$ where $\epsilon$ is the trivial line bundle then $$\secat[\ddot\xi: \ddot E(\xi)\to \dot E(\xi)]\le \secat[\dot \eta: \dot E(\eta)\to B] +1$$ and 
$$\tc_r[\dot \xi: \dot E\to B]\le \secat[\dot \eta: \dot E(\eta)\to B] + r.$$ 
\end{corollary}
\begin{corollary}\label{cor:16}
If a vector bundle $\xi: E\to B$ admits two non-vanishing linearly independent sections then 
$\secat[\ddot\xi: \ddot E\to \dot E]\le 1$ and $\tc_r[\dot \xi: \dot E\to B]\le r.$ If additionally the rank $q=\rk(\xi)$ is odd then 
$\tc_r[\dot \xi: \dot E\to B]= r.$
\end{corollary}
\begin{proof}
We apply Proposition \ref{prop:upper1} with both bundles $\dot \xi_1: \dot E(\xi_1)\to B$  and $\dot \xi_2: \dot E(\xi_2)\to B$ having globally defined continuous sections. Proposition \ref{prop:upper1} gives $\secat[\ddot \xi: \ddot E(\xi)\to \dot E(\xi)]\le 1$ and finally by Proposition \ref{prop:upper}, $\tc_r[\dot \xi: \dot E(\xi)\to B]\le r$. If $q=\rk(\xi)$ is odd then $\tc_r(S^{q-1})=r$ and hence 
$\tc_r[\dot \xi: \dot E(\xi)\to B]\ge r$. 
\end{proof}

\section{Cohomology algebra of a sphere bundle}\label{sec:6}

In this section we describe the cohomology algebra of a sphere bundle admitting a section. 
The result of this section repeats a part of the arguments of the proof of Theorem 4.1 from \cite{FW23}. 
Cohomology algebras of more general sphere bundles were studied by W. Massey in \cite{Ma}. 
All cohomology groups in this section are with integer coefficients unless the coefficients are indicated explicitly. 

\begin{theorem}\label{thm:sec}
Let $\eta: E(\eta)\to B$ be an oriented rank $q-1$ real vector bundle and let $\xi=\eta\oplus \epsilon$ be the Whitney sum of $\eta$ and the trivial line bundle $\epsilon$. The integral cohomology ring $H^\ast(\dot E(\xi))$ of the unit sphere bundle $\dot E(\xi)$ is the quotient of the polynomial ring $H^\ast(B)[\mathfrak u]$, where $\deg \u=q-1$, with respect to the principal ideal generated by the class 
\begin{eqnarray}\u^2 -\mathfrak e(\eta) \cdot \u,\end{eqnarray}
where $\e(\eta)\in H^{q-1}(B)$ denotes the Euler class of $\eta$. 
\end{theorem}
\begin{proof} Fix an orientation of $\epsilon$, i.e. select a section $s: B\to \dot E(\epsilon)\subset \dot E(\xi)$. Together with the orientation of $\eta$ this determines an orientation of $\xi$.

Since the sphere bundle $\dot \xi: \dot E(\xi)\to B$ has a section its Euler class vanishes, $\e(\xi)=0\in H^q(B)$. It
admits a cohomological extension of the fibre, i.e. a cohomology class $\u\in H^{q-1}(\dot E(\xi))$ such that its restriction onto every fibre is the fundamental class of the fibre, determined by the orientation of $\xi$.
The Leray - Hirsch theorem (see \cite{Spa66}, chapter 5, \S 7) gives a 
graded group isomorphism
$H^\ast(\dot E(\xi))\simeq H^\ast(B)\otimes H^\ast(S^{q-1}).$
In other words, every cohomology class $a\in H^\ast(\dot E(\xi))$ has a unique representation in the form 
\begin{eqnarray} \label{rep}
a= u + v\cdot \u,
\end{eqnarray}
where $u, v\in H^\ast(B)$. Here we identify $H^\ast(B)$ with its image $\dot \xi^\ast (H^\ast(B))\subset H^\ast(\dot E(\xi))$ under the ring homomorphism $\dot\xi^\ast: H^\ast(B)\to H^\ast(\dot E(\xi))$, which is a monomorphism; the symbol $v\cdot \u$ in (\ref{rep}) stands for the cup-product $\dot\xi^\ast(v)\cup \u$. 
The class $\u$ is not unique, it can be replaced by any class of the form $\u'=\u+u$ where $u\in H^{q-1}(B)$. 
We show below that there exists a specific choice of the fundamental class $\u$ such that 
\begin{eqnarray}\label{formula}
\u^2 = \e(\eta)\cdot \u \, \in \, H^{2(q-1)}(\dot E(\xi)).
\end{eqnarray}
Note that in this formula one has $ \e(\eta)\cdot \u =  \u\cdot \e(\eta)$: indeed, if $q$ is odd the classes commute but for $q$ even the Euler class $ \e(\eta)$ has order 2.  

Denote by $W\subset \dot E(\xi)$ the set of vectors $e$ such that their scalar product with the section $s(\xi(e))$ is non-negative, 
$$W=\{e\in \dot E(\xi); \, \, \langle e, s(\xi(e))\rangle \, \ge  0\}.$$
Similarly, we denote 
$$W'=\{e\in \dot E(\xi); \, \, \langle e, s(\xi(e))\rangle \, \le  0\}.$$
Each of the sets $W, W'$ is homeomorphic to the total space of the unit disc bundle of $\eta$; the intersection $$W\cap W'=\partial W=\dot E(\eta)$$ is the total space of the unit sphere bundle of $\eta$. 

Let $\u'\in H^{q-1}(\dot E(\xi))$ be an arbitrary cohomological extension of the fibre. Denote by $\sigma: B\to 
\dot E(\xi)$ the reflected section, i.e. $\sigma(b) = -s(b)$ for $b\in B$. Then the class $\u=\u'-a$, where
$a=\sigma^\ast (\u')\in H^{q-1}(B)$, is a cohomological extension of the fibre satisfying $\sigma^\ast(\u)=0$. 
We show below that such $\u$ satisfies (\ref{formula}). 

The equation $\sigma^\ast(\u)=0$ implies $\u|_{W'}=0$ and hence $\u$ can be lifted to a relative cohomology class 
$$\tilde \u\in H^{q-1}(\dot E(\xi), W')\simeq H^{q-1}(W, \partial W)\simeq H^{q-1}(W/ \partial W).$$
The long exact sequence 
$$\dots \to H^{q-2}(\dot E(\xi))\stackrel {\simeq} \to H^{q-2}(W') \to H^{q-1}(\dot E(\xi), W') \to H^{q-1}(\dot E(\xi))$$
shows that the lift $\tilde \u$ is unique. Clearly, $W/\partial W$ is the Thom space of the bundle $\eta$ and  $\tilde \u\in H^{q-1}(W/\partial W)$ is the Thom class, see \S 9 of \cite{MS}. By the definition the class 
\begin{eqnarray}
\label{euclass}
s^\ast(\tilde \u)=\e(\eta)\in H^{q-1}(B)
\end{eqnarray}
is the Euler class of the bundle $\eta$. Then in the group $H^{2(q-1)}(W, \partial W)$ we have 
$$\tilde \u\cup \tilde \u = (\tilde \u|_W) \cup \tilde \u = (\xi|_W)^\ast(\e(\eta))\cup \tilde \u.$$
Viewing this as an equality in $H^{2(q-1)}(\dot E(\xi), W')$ and applying the restriction homomorphism 
$H^{2(q-1)}(\dot E(\xi), W') \to H^{2(q-1)}(\dot E(\xi))$ gives (\ref{formula}). 
\end{proof}
For future reference we state the following Corollary which is proven by the arguments above.

\begin{corollary} \label{cor:4} Under the conditions of Theorem \ref{thm:sec}:
\begin{enumerate}\item[{(A)}] There exists a unique choice of the cohomological extension of the fibre $\u\in H^{q-1}(\dot E(\xi))$ such that 
\begin{eqnarray}\label{ext}
 s^\ast(\u) = \e(\eta)  \quad \mbox{and}\quad \u^2=\e(\eta)\cdot \u,
\end{eqnarray}
where $s: B\to \dot E(\xi)$ is the section determined by the orientation of the bundle $\epsilon$. 
\item[{(B)}]The homomorphism $s^\ast: H^\ast(\dot E(\xi)) \to H^\ast(B)$ is surjective and its kernel is the principal ideal generated by the class $\u- \e(\eta)$. 

\item[{(C)}] The length of the longest nontrivial product of classes in the ideal
$$
\ker[s^\ast: H^\ast(\dot E(\xi)) \to H^\ast(B)]
$$
equals $\h(\e(\eta))+1$ where $\h(\e(\eta))$ denotes the largest integer $k$ such that the power 
$\e(\eta)^k\in H^\ast(B)$ is nonzero. 
\end{enumerate}
\end{corollary}
\begin{proof}
(A) The left formula in (\ref{ext}) follows from (\ref{euclass}). The uniqueness of such $\u$ is obvious; it follows from the arguments of the proof of Theorem \ref{thm:sec}. The right formula in (\ref{ext}) is a part of the statement of Theorem \ref{thm:sec}.

To prove (B) consider a class $a=u+v\cdot \u$ satisfying $s^\ast(a)=0$. Since $s^\ast(u)=u$, $s^\ast(v)=v$ and 
$s^\ast(\u)=\e(\eta)$ the equality $s^\ast(a)=0$ gives $u=-v\cdot \e(\eta)$ and therefore $a=v\cdot (\u-\e(\eta)).$

Finally to prove (C) we note that the powers of the class $\u-\e(\eta)$ are given by the formula
\begin{eqnarray}\label{uek}
(\u-\e(\eta))^k = (-\e(\eta))^{k-1}\cdot (\u-\e(\eta))
\end{eqnarray}
as follows by induction from (A). 
We see that $k$-th power is nonzero for $k=\h(\e(\eta))+1$ and it vanishes for $k=\h(\e(\eta))+2$. The remaining statements are obvious consequences of statements (A) and (B). 
\end{proof}

Here is a similar statement with $\Z_2$ coefficients; note that we do not require the bundle $\xi$ to be orientable:

\begin{theorem}\label{thm:sphere2}
Let $\xi: E(\xi)\to B$ be a rank $q-1$ vector bundle such that the unit sphere bundle $\dot \xi: \dot E(\xi)\to B$ admits a continuous section $s: B\to \dot E(\xi)$. Then:
\begin{enumerate}
\item[{(a)}]
 there exists a unique cohomological extension of the fibre $\u\in H^{q-1}(\dot E(\xi);\Z_2)$ such that 
\begin{eqnarray}\label{ext2}
 s^\ast(\u) = w_{q-1}(\xi)  \quad \mbox{and}\quad \u^2=w_{q-1}(\xi)\cdot \u.
\end{eqnarray} 
Here $w_{q-1}(\xi)\in H^{q-1}(B;\Z_2)$ denotes the Stiefel - Whitney class of $\xi$; we identify $H^\ast(B;\Z_2)$ with its image under the injective homomorphism $\dot\xi^\ast: H^\ast(B;\Z_2)  \to H^\ast(\dot E(\xi);\Z_2)$.
\item[{(b)}]
The cohomology algebra $H^\ast(\dot E(\xi);\Z_2)$ is the quotient of the polynomial extension $H^\ast(B;\Z_2)[\u]$ with respect to the ideal generated by $\u^2+w_{q-1}(\xi)\cdot \u$. 
\item[{(c)}] The homomorphism $s^\ast: H^\ast(\dot E(\xi);\Z_2) \to H^\ast(B;\Z_2)$ is surjective and its kernel is the principal ideal generated by the class $\u+ w_{q-1}(\xi)$. 
\item[{(d)}] The length of the longest nontrivial product of classes in the ideal
$$
\ker[s^\ast: H^\ast(\dot E(\xi);\Z_2) \to H^\ast(B;\Z_2)]
$$
equals $\h(w_{q-1}(\xi))+1$. 
\end{enumerate}
\end{theorem}
\begin{proof}
As in the proof of Theorem \ref{thm:sec} (see also Corollary \ref{cor:4}, (C)) we may represent $\xi=\eta\oplus \epsilon$ 
where $\epsilon$ is the trivial line bundle and by repeating the arguments of the proof of Theorem \ref{thm:sec} with $\Z_2$ coefficients we obtain that the cup-length of the kernel of $s^\ast$ equals $\h(w_{q-1}(\eta))+1$. However $w_{q-1}(\eta)=w_{q-1}(\xi)$ in view of the Cartan formula. 
\end{proof}

\section{Lower bounds for the sequential parametrized topological complexity}\label{sec:7}

Let $\xi: E\to B$ be an oriented rank $q\ge 2$ vector bundle. Let $\dot \xi : \dot E\to B$ denote the unit sphere bundle of $\xi$ and let $$\ddot \xi: \ddot E\to \dot E$$ denote the associated Stiefel bundle. The total space $\ddot E$ is the space of all pairs $(e, e')\in \dot E\times_B \dot E$ of unit vectors lying in the same fibre and mutually orthogonal, $e\perp e'$; the projection is given by the formula
$\ddot \xi(e, e') = e$. The fibre of the Stiefel bundle $\ddot \xi: \ddot E\to \dot E$ is the sphere $S^{q-2}$. 

The orientation of the original bundle $\xi$ determines an orientation of the Stiefel bundle $\ddot \xi: \ddot E\to \dot E$ as follows. The bundle $\dot\xi^\ast(\xi)$ (which is induced by the map $\dot \xi$ from $\xi$) is oriented by the orientation of $\xi$. This bundle is the Whitney sum $\dot\xi^\ast(\xi) = \eta\oplus \epsilon$ where 
$$\eta=\{(e, e'); |e|=1, \, \xi(e)=\xi(e'), \, e\perp e'\}$$ and $\epsilon$ is the trivial line bundle whose fibre over a unit vector $e$ is the set of all real multiples of $e$. Clearly, $\dot \eta=\ddot \xi$. The line bundle $\epsilon$ is canonically oriented by the section $s(e)=(e, e)$ where $e\in \dot E$. Hence we see that an orientation of $\xi$ determines an orientation of $\eta$ and of its unit sphere bundle $\ddot \xi: \ddot E\to \dot E$. Therefore, the Euler class $\e(\ddot \xi)\in H^{q-1}(\dot E;\Z)$ is well-defined. 

\begin{theorem}\label{thm:lower} In the notations introduced in \S\ref{sec:seqpar}, consider 
the space $\dot E^r_B$, where $r\ge 2$, and the diagonal map $\Delta: \dot E\to \dot E^r_B$, where $\Delta(e)=(e, e, \dots, e)$ for $e\in \dot E$. 
Then the cup-length of the kernel 
\begin{eqnarray}\label{kernel}
\ker[\Delta^\ast: H^\ast(\dot E^r_B;\Z) \to H^\ast(\dot E;\Z)]\end{eqnarray} 
equals $\h(\e(\ddot\xi))+r-1.$
\end{theorem}
Here the symbol $\h(\e(\ddot \xi))$ denotes {\it the height of the Euler class} $\e(\ddot \xi)$, i.e. the smallest integer $k\ge 0$ such that the power 
$\e(\ddot \xi)^{k+1}=0\, \in H^{(q-1)k}(\dot E)$ vanishes. In particular, $\h(\e(\ddot \xi))=0$ if $\e(\ddot \xi)=0$.

The proof of Theorem \ref{thm:lower} will occupy the rest of this section. We shall skip the coefficients $\Z$ from the notations. 
 
For any $1\le j\le r$ consider the space $\dot E^j_B\subset \dot E^j= \dot E\times \dot E\times \dots \times \dot E$ \, \, ($j$ times) consisting of the $j$-tuples $(e_1, \dots, e_j)$ of unit vectors lying in the same fibre, i.e. $\xi(e_1)=\xi(e_2)=\dots=\xi(e_j)$. 
For $1\le i\le j\le r$ there is the projection $\pi^j_i: \dot E^j_B\to \dot E^{i}_B$ given by 
$$\pi^j_i(e_1, \dots, e_j)= (e_1, e_2, \dots, e_{i}),$$ 
i.e. the map $\pi^j_i$\,  \lq\lq forgets\rq\rq \, all the components $e_k$ with $i<k\le j$. 
Besides, for $1\le i\le j\le r$ there is the map 
$\sigma^i_j: \dot E^i_B \to \dot E^j_B$ given by
$$
\sigma^i_j(e_1, e_2, \dots, e_i) =(e_1, e_2, \dots, e_{i-1}, e_i, e_i, \dots, e_i),
$$
i.e. the last vector $e_i$ is repeated $j-i$ times. 
Clearly, for $i\le j$ one has $\pi^j_i\circ \sigma^i_j = 1_{\dot E^i_B}$, i.e. $\sigma^i_j$ is a section of $\pi^j_i$. 
Additionally, for $i\le j\le k\le r$ one has 
$$
\pi^j_i\circ \pi^k_j =\pi^k_i \quad \mbox{and}\quad \sigma^j_k\circ \sigma^i_j=\sigma^i_k.
$$


We have the following tower of fibrations and their sections
\begin{eqnarray}\label{tower}
\hskip 1cm
\xymatrix{
\dot E\ar@<1ex>[r]^{\sigma^1_2} 
& \dot E^2_B\ar@<1ex>[l]^{\pi^2_1}\ar@<1ex>[r]^{\sigma^2_3}
& \dot E^3_B\ar@<1ex>[l]^{\pi^3_2}\ar@<1ex>[r] &{}\ar@<1ex>[l]& \dots 
&{} \ar@<1ex>[r]^{\sigma^{r-3}_{r-2}}
&\dot E^{r-2}_B \ar@<1ex>[r]^{\sigma^{r-2}_{r-1}}\ar@<1ex>[l]^{\pi^{r-2}_{r-3}}
&\dot E^{r-1}_B\ar@<1ex>[l]^{\pi^{r-1}_{r-2}}\ar@<1ex>[r]^{\sigma^{r-1}_r}
& \dot E^r_B\ar@<1ex>[l]^{\pi^{r}_{r-1}}.
}
\end{eqnarray}
Each map $\pi_{j-1}^j$ is a locally trivial bundle with fibre the sphere $S^{q-1}$. 
The composition 
\begin{eqnarray}\label{comp}
\Delta \, = \, \sigma^{r-1}_r\circ \sigma^{r-2}_{r-1}\circ \dots\circ \sigma^1_2\, : \, \dot E \to \dot E^r_B 
\end{eqnarray}
coincides with the diagonal map $\Delta: \dot E\to \dot E^r_B$, \, \, $\Delta(e)=(e, e, \dots, e). $

Theorem \ref{thm:sec} is applicable to each level of the tower (\ref{tower}); we see that every homomorphism 
\begin{eqnarray}
{(\pi^{i+1}_{i})}^\ast : H^\ast(\dot E^{i}_B) \to H^\ast (\dot E_B^{i+1}), \quad i= 1, 2,  \dots, r-1
\end{eqnarray}
is a monomorphism 
and the cohomology algebra 
$H^\ast(\dot E^{i+1}_B)$ is the quotient of the polynomial extension $H^\ast(\dot E^{i}_B)[\u_i]$ by the ideal generated by the polynomial $$\u_i^2-\e(\eta_i)\cdot \u_i.$$ The classes 
$$\u_i\in H^{q-1}(\dot E^{i+1}_B) \quad \mbox{and}\quad \e(\eta_i)\in H^{q-1}(\dot E^{i}_B)$$
satisfy
\begin{eqnarray}\label{26}
{(\sigma^{i}_{i+1})}^\ast(\u_i) = \e(\eta_i), \quad i=1, 2, 3, \dots, r-1.
\end{eqnarray}
Here $\eta_i: \dot E(\eta_i) \to \dot E^{i}_B$ is the bundle of spheres of dimension $q-2$ over 
$\dot E^{i}_B$, it is defined as the bundle of vectors orthogonal to the section $\sigma^{i}_{i+1}$ (see Theorem \ref{thm:sec}). 
 The total space $\dot E(\eta_i) $ consists of $(i+1)$-tuples
$(e_1, \dots, e_{i}, e_{i+1})$ where $e_j\in \dot E(\xi)$ are unit vectors satisfying $\xi(e_1)=\xi(e_2)=\dots=\xi(e_i)=\xi(e_{i+1})$ and $e_{i}\perp e_{i+1}$. 
The map $\eta_i$ acts by $\eta_i(e_1, \dots, e_{i}, e_{i+1})= (e_1, \dots, e_{i})$. 

Let the map $h_{i}: \dot E^{r}_B \to \dot E$ be given by $h_{i}(e_1, e_2, \dots, e_{r}) = e_{i}$ where $i=1, 2, 3, \dots, r$, and let 
$\eta'_i$ denote the induced bundle over $\dot E^r_B$, i.e. $\eta'_i = h_i^\ast(\ddot\xi)$. 
Then clearly,
\begin{eqnarray}\label{28}
\eta_i = (\sigma^{i}_r)^\ast (\eta'_{i})\quad \mbox{and therefore}\quad \e(\eta_i) = (\sigma^{i}_r)^\ast(\e(\eta'_{i})).
\end{eqnarray}
Here $\e(\eta'_i)\in H^{q-1}(\dot E^r_B)$ stands for the Euler class of $\eta'_i$. 
We claim that 
\begin{eqnarray}\label{deltaeta}
\Delta^\ast(\e(\eta'_i)) =\e(\ddot \xi) \, \, \in\, H^{q-1}(\dot E), \quad i=1, 2, 3, \dots, r. 
\end{eqnarray}
Indeed, the composition 
$$
\dot E\stackrel \Delta \to \dot E^r_B\stackrel{h_i}\to \dot E
$$
is the identity map and therefore $\Delta^\ast(\eta'_i) =\Delta^\ast(h_i^\ast(\ddot\xi)) = \ddot \xi$ which implies (\ref{deltaeta}) due to functoriality of the Euler classes. 

We shall denote
$$
\v_i = (\pi_{i+1}^r)^\ast(\u_i)\, \in H^{q-1}(\dot E^r_B), \quad \mbox{where}\quad i=1, 2, \dots, r-1.
$$
We claim that 
\begin{eqnarray}\label{30a}
\Delta^\ast(\v_i) =\e(\ddot \xi) \end{eqnarray}
and
\begin{eqnarray}\label{30b} \v_i^2=\e(\eta'_i)\cdot \v_i.
\end{eqnarray}
To show (\ref{30a}) we note
\begin{eqnarray*}
\Delta^\ast(\v_i) &=& \Delta^\ast ((\pi^r_{i+1})^\ast(\u_i))= (\sigma^1_{i+1})^\ast(\u_i)= (\sigma^1_i)^\ast((\sigma^i_{i+1})^\ast(\u_i))\\
&=& (\sigma^1_i)^\ast(\e(\eta_i)) = (\sigma^1_i)^\ast((\sigma^i_r)^\ast(\e(\eta'_i))) =\Delta^\ast(\e(\eta'_i))=\e(\ddot\xi).
\end{eqnarray*}
In this calculation we used (\ref{26}), (\ref{28}) and (\ref{deltaeta}). 

To show (\ref{30b}) we note that 
\begin{eqnarray}\label{31}
\u_i^2 = (\pi^{i+1}_i)^\ast(\e(\eta_i))\cdot \u_i\, \in \, H^{q-1}(\dot E^{i+1}_B)
\end{eqnarray}
(see Theorem \ref{thm:sec}). Applying the homomorphism $(\pi^r_{i+1})^\ast$ to both sides of (\ref{31}) and noting that $(\pi^r_{i+1})^\ast(\u_i)=\v_i$
and 
\begin{eqnarray*}
(\pi^r_{i+1})^\ast((\pi^{i+1}_i)^\ast(\e(\eta_i)))&=& (\pi^r_i)^\ast(\e(\eta_i))= (\pi^r_i)^\ast((\sigma^i_r)^\ast(\e(\eta'_i)))\\ &=&
(\pi^r_i)^\ast((\sigma^i_r)^\ast(h_i^\ast(\e(\ddot\xi))))
= h_i^\ast(\e(\ddot\xi))\\ &=&
\e(\eta'_i),
\end{eqnarray*}
we obtain (\ref{30b}). On the last line of the above calculation we used the equality $h_i\circ \sigma^i_r\circ \pi^r_i=h_i$.

\begin{lemma} The kernel $\ker[\Delta^\ast: H^\ast(\dot E^r_B) \to H^\ast(\dot E)]$ coincides with the ideal of the algebra $H^\ast(\dot E^r_B)$ generated by the classes
\begin{eqnarray}\label{classes}
\v_1 -\e(\eta'_1), \ \v_2-\e(\eta'_2), \ \dots, \ \v_{r-1}-\e(\eta'_{r-1}).
\end{eqnarray}
\end{lemma}
\begin{proof} Comparing formulae (\ref{deltaeta}) and (\ref{30a}) we see that all classes (\ref{classes}) lie in the kernel of $\Delta^\ast$. Hence the ideal generated by these classes is contained in the kernel of $\Delta^\ast$ and we only need to establish 
the opposite inclusion. 

We denote $A_0=(\pi^r_1)^\ast(H^\ast(\dot E))\subset H^\ast(\dot E^r_B)$ and
for $i=1, 2, \dots, r-1$ we denote by 
$A_i\subset H^\ast(\dot E^r_B)$ the subalgebra generated by $A_0=(\pi^r_1)^\ast(H^\ast(\dot E))$ and the classes $\v_j$ with $j\le i$. Equivalently, $A_j=(\pi^r_{j+1})^\ast(H^\ast(\dot E^{j+1}_B))$. 

The class $\e(\eta'_j)$ lies in the subalgebra $A_{j-1}$. Indeed, the map $h_j: \dot E^r_B\to \dot E$ can be decomposed as $h_j = h'_j\circ \pi^r_j$ where $h'_j: \dot E^j_B\to \dot E$ is given by $h'_j(e_1, \dots, e_j)=e_j$. Thus, we obtain
$$\e(\eta'_j) = h_j^\ast(\e(\ddot \xi))= (\pi^r_j)^\ast((h'_j)^\ast(\e(\ddot\xi))\, \in \, A_{j-1}.
$$

We show by induction that $(\ker \Delta^\ast) \cap A_i$ is contained in the ideal of the algebra $A_i$ generated by the classes
$\v_j-\e(\eta'_j)$ with $j\le i$. 
Clearly, $(\ker \Delta^\ast) \cap A_0=0$ since the composition $\Delta^\ast \circ (\pi^r_1)^\ast $ is the identity map $H^\ast(\dot E)\to H^\ast(\dot E)$. 
Next we shall assume that $(\ker \Delta^\ast) \cap A_{i-1}$ is contained in the ideal of $A_{i-1}$ generated by the classes
$\v_j-\e(\eta'_j)$ with $j\le i-1$. 
Any class $x\in (\ker \Delta^\ast) \cap A_i$ can be uniquely represented in the form 
$$x=a+b\cdot \v_i= \left[a+b\cdot \e(\eta'_i)\right] +b\cdot (\v_i-\e(\eta'_i)$$ 
where $a, b\in A_{i-1}$.
Then $a+b\cdot \e(\eta'_i)\in (\ker \Delta^\ast) \cap A_{i-1}$, and our statement follows by induction. 
\end{proof}
%
%
%
%
%

 As a consequence of the above result we obtain:

\begin{corollary}\label{cor:22}
The cup-length of the ideal (\ref{kernel}) coincides with the minimal number 
$N$ such that the product 
\begin{eqnarray}
\label{prod22}
(\v_1-\e(\eta'_1))^{\alpha_1}\cdot (\v_2-\e(\eta'_2))^{\alpha_2}\cdot \dots \cdot (\v_{r-1}-\e(\eta'_{r-1}))^{\alpha_{r-1}}\, =\, 0\, \in H^\ast(\dot E^r_B)\end{eqnarray}
vanishes for all integers $\alpha_1, \alpha_2 \dots, \alpha_{r-1}\ge 0$ satisfying $\sum_{i=1}^{r-1}\alpha_i\ge N+1$. 
\end{corollary}

Clearly, here without loss of generality we may assume that each $\alpha_i$ satisfies $\alpha_i\ge 1$ since any nonzero product as above with $\alpha_i=0$ can be made longer and still nonzero by multiplying it by the factor $\v_i-\e(\eta'_i)$. We denote
$$\mathfrak U = (\v_1-\e(\eta'_1))\cdot (\v_2-\e(\eta'_2))\cdot \dots (\v_{r-1}-\e(\eta'_{r-1})) \, \in \, H^{(q-1)(r-1)}(\dot E^r_B).$$
The class $\U$ satisfies 
\begin{eqnarray}\label{zero}
\v_i\cdot \U=0\quad \mbox{ for every}\quad  i=1, 2, \dots, r-1,
\end{eqnarray} as follows from (\ref{30b}). Thus, Corollary \ref{cor:22} gives:

\begin{corollary}\label{cor:23}
The cup-length of the ideal (\ref{kernel}) equals 
$r-1+ M $
where $M\ge 0$ is the minimal integer
such that the product 
\begin{eqnarray}
\U \cdot \prod_{i=1}^{r-1}\e(\eta'_i)^{\beta_i} \ =0\, \in H^\ast(\dot E^r_B).
\end{eqnarray}
vanishes for all $(\beta_1, \dots, \beta_{r-1})$, satisfying  $\beta_i\ge 0$ and $\sum_{i=1}^{r-1} \beta_i> M$. 
\end{corollary}

Next we find explicitly the classes $\e(\eta'_i)\in H^{q-1}(\dot E^r_B).$ One can write
\begin{eqnarray}\label{etaprime}
\e(\eta'_i) = (\pi^r_1)^\ast(x_i)+ \sum_{j=1}^{r-1} n_{ij}\v_j, \quad \mbox{where}\quad x_i\in H^{q-1}(\dot E), \quad n_{ij}\in \Z,
\end{eqnarray}
since every class in $H^{q-1}(\dot E^r_B)$ has a unique representation in this form. Our goal below is to determine the class $x_i$ and the integers $n_{ij}$.  

For $j=1, 2, 3, \dots, r$ consider the subset $Z_j \subset (\dot E_b)^r \subset \dot E^r_B$ 
where $b\in B$ is a fixed point of the base and
$Z_j = X_1 \times X_2\times \dots\times X_r$ with $X_k=\{e_k^0\}$ (a fixed single point $e^0_k\in \dot E_b$) for 
 $k\in \{1, 2, \dots, j-1, j+1, \dots, r\}$ and the factor $X_j$ equals the whole $\dot E_b$, i.e. $X_j=\dot E_b$.
 Clearly $Z_j$ is homeomorphic to the sphere of dimension $q-1$, which is oriented by the orientation of the original bundle $p: E\to B$. 

Consider the restriction of the bundle $\eta'_i$ onto the sphere $Z_j$. By definition $\eta'_i=h_i^\ast(\ddot \xi)$ (see above); hence we obtain that the restriction $\eta'_i$ onto $Z_j$ is a trivial bundle for $j\not=i$ and $\eta'_i|_{Z_i}$ is the unit tangent bundle of the sphere $Z_i$. We know that the evaluation of the Euler class of the sphere on its fundamental class equals the Euler characteristic of the sphere $S^{q-1}$, i.e. $0$ if $q$ is even and $2$ is $q$ is odd. Therefore we obtain:

\begin{corollary}\label{cor:24}
Evaluation of the Euler class $\e(\eta'_i)$ on the homology class $[Z_j]\in H_{q-1}(\dot E^r_B)$ equals $0$ if either $q$ is even or $j\not=i $. For $q$ odd the evaluation of the class $\e(\eta'_i)$ on the class $[Z_i]$ equals $2$. 
\end{corollary}

This implies:

\begin{corollary}\label{cor:25}
If $q$ is even then 
\begin{eqnarray}
\e(\eta'_i)=(\pi^r_1)^\ast(\e(\ddot \xi))\quad \mbox{for every} \quad i=1, 2, \dots, r-1.
\end{eqnarray} 
\end{corollary}
\begin{proof} For $i=1, 2, \dots, r-1$ and $j=2, \dots, r$ one has
$$\langle \v_i, [Z_j]\rangle =\left\{
\begin{array}{lll}
1, & \mbox{if} & j=i+1,\\
0, & \mbox{if}& j\not=i+1
\end{array}
\right.$$
and besides, $\langle (\pi^r_1)^\ast x, [Z_j]\rangle =0$ for every $x\in H^{q-1}(\dot E)$ and $j\ge 2$. Thus, evaluating (\ref{etaprime}) on the class $[Z_j]$ we obtain using Corollary \ref{cor:24} that $n_{ij}=0$ for all $i$ and $j$. The result now follows from (\ref{deltaeta}).
\end{proof}
\begin{corollary}\label{cor:26}
If $q$ is odd then 
\begin{eqnarray}\label{39}
\e(\eta'_i)=2\cdot \v_i - (\pi^r_1)^\ast(\e(\ddot \xi))
\end{eqnarray} 
for every $i=1, 2, \dots, r-1$. 
\end{corollary}
\begin{proof}
The proof is similar to the proof of the previous Corollary. We evaluate (\ref{etaprime}) on the homology classes $[Z_2], \dots, [Z_r]$ and using Corollary \ref{cor:24} we obtain that for $q$ odd the coefficients $n_{ij}$ in (\ref{etaprime}) are given by 
$n_{ij}=2\delta_{ij}$. Finally, using (\ref{deltaeta}) and (\ref{30a}) the result (\ref{39}) follows. 
\end{proof}
\begin{proof}[Proof of Theorem \ref{thm:lower}] We start with the result of Corollary \ref{cor:23} and note that for $q$ even using Corollary \ref{cor:25} one has
$$\U \cdot \prod_{i=1}^{r-1}\e(\eta'_i)^{\beta_i} =\U \cdot (\pi^r_1)^\ast\left(\e(\ddot\xi)^{\sum_{i=1}^{r-1}\beta_i}\right).$$
This product vanishes if and only if $\e(\ddot\xi)^{\sum_{i=1}^{r-1}\beta_i}=0$, i.e. iff $\sum_{i=1}^{r-1} \beta_i \, > \, \h(\e(\ddot \xi))$. 

For $q$ odd the argument is similar but instead of Corollary \ref{cor:25} we use Corollary \ref{cor:26}. Taking into account (\ref{zero}) we obtain
$$\U \cdot \prod_{i=1}^{r-1}\e(\eta'_i)^{\beta_i} =\, \pm \, \U \cdot (\pi^r_1)^\ast\left(\e(\ddot\xi)^{\sum_{i=1}^{r-1}\beta_i}\right)$$
and as above this product vanishes if and only if $\sum_{i=1}^{r-1} \beta_i \, > \, \h(\e(\ddot \xi))$.
\end{proof}
\begin{corollary}\label{cor:lbound}
The sequential parametrized topological complexity of the unit sphere bundle associated with an oriented vector bundle $\xi: E\to B$ is bounded below by $\h(\e(\ddot \xi))+r-1$, i.e.
\begin{eqnarray}
\tc_r[\dot \xi: \dot E\to B] \ge \h(\e(\ddot \xi))+r-1.
\end{eqnarray}
\end{corollary}
\begin{proof}
This follows from Proposition \ref{lem:lowerbound} and Theorem \ref{thm:lower}. 
\end{proof}

\section{A few special cases}\label{sec:8}

First we consider situations when the dimension of the base $B$ is small. 

\begin{corollary}\label{cor:equal}
Let $\xi: E\to B$ be a rank $q$ vector bundle where the base $B$ is a CW-complex of dimension 
\begin{eqnarray}\label{ineqdim}
\dim B \le (q-1)\cdot \h(\e(\ddot \xi)).
\end{eqnarray}
 Then for any $r=2, 3, \dots$ one has
\begin{eqnarray}\label{eqqq}
\tc_r[\dot \xi: \dot E \to B] = \h(\e(\ddot \xi))+r-1 = \secat[\ddot \xi: \ddot E\to \dot E]+r-1. 
\end{eqnarray}
\end{corollary}
\begin{proof}
Applying Lemma 2.1(c) from \cite{FW23} we see that our assumptions imply that $\h(\e(\ddot \xi))=\secat[\ddot \xi: \ddot E\to \dot E]$. 
Now Proposition \ref{prop:upper} combined with Corollary \ref{cor:lbound} 
give our claim (\ref{eqqq}). 
\end{proof}

\begin{corollary} Let $\xi: E\to B$ be a rank $q\ge 2$ vector bundle with the base $B$ a CW-complex. Assume that 
$q$ is odd and  satisfies 
\begin{eqnarray}\label{ineq:2}
\dim B \le q -1.\end{eqnarray}
Then 
$\tc_r[\dot \xi: \dot E \to B]$ equals either $r$ or $r+1$.  Moreover, if $q$ is odd and $\dim B< q-1$ then $\tc_r[\dot \xi: \dot E \to B]=r$.
\end{corollary}
\begin{proof}
If $q$ is odd then the Euler class $\e(\ddot \xi)\in H^{q-1}(\dot E)$ is nonzero since its restriction onto each fibre of the bundle 
$\dot \xi: \dot E\to B$ is nonzero [as it equals the Euler characteristic of the sphere $S^{q-1}$]. 
Thus, in this case $\h(\e(\ddot\xi))\ge 1$ and (\ref{ineq:2}) implies (\ref{ineqdim}) and the equalities
\begin{eqnarray}\label{30}
\tc_r[\dot \xi: \dot E \to B] = \h(\e(\ddot \xi))+r-1 = \secat[\ddot \xi: \ddot E\to \dot E]+r-1
\end{eqnarray}
follow from Corollary \ref{cor:equal}. 

On the other hand, since $\dim \dot E =\dim B +q-1\le 2(q-1)$ we see that 
$1\le \h(\e(\ddot \xi))\le 2$ and therefore (\ref{30}) implies that $\tc_r[\dot \xi: \dot E \to B]$ equals either $r$ or $r+1$. The proof of the last statement is based on the fact that under the indicated assumptions $\dim \dot E <2(q-1)$ and hence $\h(\e(\ddot \xi))=1$. 
\end{proof}

Next we give an application of the sharp upper bound of Theorem \ref{sharp} combined with Theorem \ref{thm:lower}:

\begin{theorem}\label{sharp3} Let $\xi: E\to B$ be an oriented rank $q\ge 3$ vector bundle. 
Suppose that $B$ is a simply connected finite CW-complex and its dimension $\dim B$ is divisible by $q-1$ and satisfies 
\begin{eqnarray} \label{hsmall}
\h(\e(\ddot \xi))\le \frac{\dim B}{q-1}.\end{eqnarray} 
Then 
\begin{eqnarray}\label{sharp4}
\tc_r[\dot\xi: \dot E\to B] \le r-1 + \frac{\dim B}{q-1}.
\end{eqnarray}
\end{theorem}
Note that the upper bound (\ref{sharp4}) is by one better than (\ref{upper2}). 
\begin{proof} Theorem \ref{sharp} part (B) is applicable. In our case $k=q-2$ and $m=r+\frac{\dim B}{q-1}$. The assumption (\ref{hsmall}) can be written in the form $\h(\e(\ddot \xi)) +r-1 <m$,
which, by Theorem \ref{thm:lower}, means that the cup-length of the kernel (\ref{ker}) is smaller than $m$. 
Now Corollary \ref{sharp1}  applies and gives 
$\tc_r[\dot\xi: \dot E\to B]<m$, which is equivalent to (\ref{sharp4}). 
\end{proof}
%
%
Next we consider the special case when the original bundle $\xi$ admits a nonzero section:

\begin{proposition}\label{thm:31}
Assume that $\xi=\eta\oplus \epsilon$ where $\eta: E(\eta)\to B$ is an oriented rank $q-1$ vector bundle and $\epsilon$ is the 
 trivial oriented line bundle. 
Then: 
(a) if $q\ge 2$ is even then 
\begin{eqnarray}\label{qeven}
\h(\e(\ddot\xi))=\h(\e(\eta));
\end{eqnarray} 
(b) if $q\ge 3$ is odd and $H^\ast(B;\Z)$ has no 2-torsion then 
\begin{eqnarray}\label{qodd}
\h(\e(\ddot\xi))=2\cdot \left\lfloor\frac{\h(\e(\eta))}{2}\right\rfloor+1.
\end{eqnarray}
In particular, we see that for $q\ge 3$ odd the height $\h(\e(\ddot\xi))$ is always odd. 
\end{proposition}

\begin{proof} Below all cohomology groups are understood with integer coefficients. We apply Theorem \ref{thm:sec} 
to the bundle $\xi: E(\xi)\to B$ to conclude that $H^\ast(\dot E(\xi))$ is the quotient of the polynomial extension $H^\ast(B)[\u]$ with respect to the principal ideal generated by the class $\u-\e(\eta)$. We shall identify the cohomology of the base $H^\ast(B)$ with its image in $H^\ast(\dot E(\xi))$ under the monomorphism $\dot \xi^\ast$.
The class $\u$ is a cohomological extension of the fibre, i.e. its restriction to each fibre is the fundamental class of the fibre. 
By Leray - Hirsch theorem, the Euler class $\e(\ddot \xi)\in H^{q-1}(\dot E(\xi))$ has a unique representation in the form
\begin{eqnarray}\label{abu}
\e(\ddot \xi) = a+ b\cdot \u,
\end{eqnarray}
where $a\in H^{q-1}(B)$ and $b\in \Z$. The restriction of the class $\e(\ddot \xi)$ to each fibre of the bundle 
$\dot \xi: \dot E\to B$ equals $\chi(S^{q-1})$ times the fundamental class of the fibre. Hence, we obtain that in (\ref{abu}) 
 $b=0$, if $q$ is even, and $b=2$ if $q$ is odd. 

Let $s:B \to \dot E(\xi)$ be the section determined by the trivial summand $\epsilon$. Geometrically it is obvious that 
$$s^\ast(\ddot \xi) = \dot \eta.$$ Therefore, using functoriality of the Euler class, we have $s^\ast(\e(\ddot \xi))=\e(\eta).$ On the other hand, 
we know that $s^\ast(\u) = \e( \eta)$, see Corollary \ref{cor:4}. Applying $s^\ast$ to both sides of equation (\ref{abu}) and noting that $s^\ast(a)=a$ we find that 
\begin{eqnarray}\label{abu2}
\e(\ddot \xi) = \left\{
\begin{array}{ll}
\e(\eta), & \mbox{if $q$ is even}, \\ 
-\e(\eta) +2\cdot \u, & \mbox{if $q$ is odd}.
\end{array}
\right.
\end{eqnarray}
From this, clearly (\ref{qeven}) follows for $q$ even. 

Consider now the case when $q\ge 3$ is odd. Using (\ref{abu2}) and the equality $\u^2=\e(\eta)\cdot \u$ (see Corollary \ref{cor:4}), we find that 
$\e(\ddot \xi)^2 = \e(\eta)^2$ and therefore we get
$$
\e(\ddot\xi)^{2n} = \e(\eta)^{2n} \quad \mbox{and}\quad \e(\ddot\xi)^{2n+1} = -\e(\eta)^{2n+1} + 2\e(\eta)^{2n}\cdot \u
$$
for any integer $n\ge 0$. Using our assumption about the absence of 2-torsion in integral cohomology of $B$,
we obtain from these equations that  
$$
\h(\e(\ddot \xi)) = \left\{
\begin{array}{llll} 
\h(\e(\eta)), &\mbox{if}& \h(\e(\eta)) &\mbox{is odd},\\
\h(\e(\eta))+1, &\mbox{if}& \h(\e(\eta)) &\mbox{is even.}
\end{array}
\right.
$$
This is equivalent to our claim (\ref{qodd}). 
\end{proof}

\begin{example}\label{ex2} Consider the canonical complex line bundle $\eta: E(\eta)\to \CP^n$ over the complex projective space. Viewing $\eta$ as a rank 2 real vector bundle we may apply Proposition \ref{thm:31} to the bundle $$\xi=\eta\oplus \epsilon.$$ Here $q=3$ is odd and $\h(\e(\eta))=n$. By Proposition \ref{thm:31} we obtain $$\h(\e(\ddot \xi))= 2\left \lfloor n/2\right\rfloor+1.$$
Applying Corollary \ref{cor:lbound} we obtain
$
\tc_r[\dot \xi: \dot E(\xi)\to \CP^n] \ge 2\left \lfloor n/2\right\rfloor+r.
$
On the other hand the inequality (\ref{upper2}) gives
$
\tc_r[\dot \xi: \dot E(\xi)\to \CP^n] \le n+r.
$
We see that the lower and upper bounds agree if $n$ is even. 

If $n$ is odd inequality (\ref{hsmall}) is satisfied and we can apply Theorem \ref{sharp3}. Inequality (\ref{sharp4}) then gives
$
\tc_r[\dot \xi: \dot E(\xi)\to \CP^n] \le n+r-1.
$
Thus, we obtain:
\begin{corollary}\label{cor:36}
For the vector bundle $\xi$ of Example \ref{ex2} one has 
\begin{eqnarray}
\tc_r[\dot \xi: \dot E(\xi)\to \CP^n] = \left\{
\begin{array}{ll}
n+r, &\mbox{if $n$ is even},\\
n+r-1, &\mbox{if $n$ is odd}.
\end{array}
\right.
\end{eqnarray}
\end{corollary}
\end{example}
\begin{example}
For comparison with the previous example consider the vector bundle 
$$\xi_\ell=\eta\oplus \ell\cdot \epsilon=\eta\oplus \epsilon\oplus \epsilon \oplus \dots\oplus\epsilon$$ 
over $\CP^n$, where $\ell =0, 1, 2, \dots$. Here $\eta$ is the canonical complex line bundle over $\CP^n$ 
The previous example corresponds to the case $\ell=1$. The rank of $\xi_\ell$ equals $q= 2+\ell$. 
If $\ell$ is even then $\xi_\ell$ admits a complex structure and hence $\tc_r[\dot\xi:\dot E\to B]=r-1$, according to Corollary \ref{cor:complex}.
If $\ell$ is odd and $\ell>1$ then $\xi_\ell$ admits 2 linear independent sections and hence by Corollary \ref{cor:16} we obtain 
 $\tc_r[\dot\xi:\dot E\to B]=r$. 
 Thus we see that only the case $\ell=1$ leads to high topological complexity as described in Corollary \ref{cor:36}.
\end{example}

\section{Lower bounds using Stiefel - Whitney characteristic classes}\label{sec:9}

In this section we state an analogue of Theorem \ref{thm:lower} in which instead of the Euler class feature the Stiefel - Whitney classes. Compared to Theorem \ref{thm:lower} this result has two advantages: firstly, it involves characteristic classes of the original vector bundle $\xi: E\to B$ and, secondly, it does not require the bundle $\xi$ to be orientable. The case $r=2$ appears in the paper \cite{FarW}. 

\begin{theorem}\label{thm:lower2} Let $\xi: E\to B$ be a locally trivial vector bundle of rank $q\ge 2$. Consider the associated sphere bundle $\dot \xi: \dot E\to B$, the space $\dot E^r_B$ and the diagonal map $\Delta: \dot E\to \dot E^r_B$. 
Then the cup-length of the kernel 
\begin{eqnarray}\label{kernel22}
\ker[\Delta^\ast: H^\ast(\dot E^r_B;\Z_2) \to H^\ast(\dot E;\Z_2)]\end{eqnarray} 
equals $$\h(w_{q-1}(\xi)|w_q(\xi))+r-1.$$
Here $w_{q-1}(\xi)\in H^{q-1}(B;\Z_2)$ and $w_{q}(\xi)\in H^{q}(B;\Z_2)$ denote the Stiefel - Whitney classes of $\xi$ and the symbol 
$\h(w_{q-1}(\xi)|w_q(\xi))$ denotes the smallest integer $k\ge 0$ such that the power $w_{q-1}(\xi)^{k+1}$ lies in the ideal of the algebra $H^\ast(B;\Z_2)$ generated by the class $w_q(\xi)$. In particular, using Proposition \ref{lem:lowerbound}, one obtains the inequality
\begin{eqnarray}\label{lbound2}
\tc_r[\dot \xi: \dot E\to B] \ge \h(w_{q-1}(\xi)|w_q(\xi))+r-1.
\end{eqnarray}
\end{theorem}
\begin{proof}  Repeating the arguments of the proof of Theorem \ref{thm:lower} with $\Z_2$ coefficients and noting that 
the reduction mod 2 of the Euler class $\e(\ddot \xi)$ equals the Stiefel - Whitney (SW) class $w_{q-1}(\ddot \xi)$, one obtains that the 
cup-length of the kernel (\ref{kernel22}) equals $\h(w_{q-1}(\ddot \xi))+r-1$. Finally we show that 
\begin{eqnarray}\label{equal222}
\h(w_{q-1}(\ddot \xi))= \h(w_{q-1}(\xi)|w_q(\xi)).
\end{eqnarray}
To prove (\ref{equal222}) we note that the bundle $\dot \xi^\ast(\xi)$ over $\dot E$ induced by the map $\dot \xi: \dot E\to B$ from $\xi$ has the form 
$\alpha\oplus \epsilon$ where $\epsilon$ is the trivial line bundle and the fibre of $\alpha$ over a point $e\in \dot E$ is $e^\perp$, i.e. the space of vectors of $E$ orthogonal to $e$. Clearly, one has $\ddot \xi= \dot \alpha$ for the unit sphere bundles and using the standard properties of the SW-classes we obtain 
\begin{eqnarray}\label{swclasses}
w_{q-1}(\ddot \xi) = \dot \xi^\ast(w_{q-1}(\xi)).
\end{eqnarray}
From the spectral sequence of fibration $\dot \xi: \dot E\to B$ we see that the kernel of the homomorphism 
$\dot \xi^\ast: H^\ast(B;\Z_2) \to H^\ast(\dot E; \Z_2)$ is the principal ideal generated by the class $w_q(\xi)\in H^q(\dot E;\Z_2)$. 
Thus, taking into account (\ref{swclasses}) we obtain that a power $w_{q-1}(\ddot \xi)^k$ vanishes if and only if the power 
$w_{q-1}(\xi)^k$ lies in the ideal of $H^\ast(B;\Z_2)$ generated by $w_q(\xi)$. This proves (\ref{equal222}). 
\end{proof}

\begin{example} For an integer $\ell=1, 2, \dots$ consider the bundle $\xi_\ell=\ell \eta\oplus \epsilon$ over ${\RP}^n$, where 
$\eta$ is the canonical line bundle over $\RP^n$ and $\epsilon$ is the trivial line bundle. The rank of this bundle equals $q=\ell+1$ and the total Stiefel - Whitney class is $(1+\alpha)^\ell$ where $\alpha\in H^1(\RP^n;\Z_2)$ is the generator. Thus we have 
$w_{q-1}(\xi_\ell) = \alpha^\ell$ and $w_q(\xi_\ell)=0$. The relative height $\h(w_{q-1}(\xi_\ell)|w_q(\xi_\ell))$ is the smallest $k\ge 0$ such that $\ell(k+1) \ge n+1$, i.e. 
$$\h(w_{q-1}(\xi_\ell)|w_q(\xi_\ell))=\left\lceil \frac{n+1-\ell}{\ell}\right\rceil= 
\left\lceil \frac{n+1}{\ell}\right\rceil-1.$$
Using (\ref{lbound2}) we obtain
$$
\tc_r[\dot \xi_\ell: \dot E(\xi_\ell) \to \RP^n] \ge \left\lceil \frac{n+1}{\ell}\right\rceil +r-2. 
$$
On the other hand, using (\ref{upper2}) we get
$$
\tc_r[\dot \xi_\ell: \dot E(\xi_\ell) \to \RP^n] \le \left\lceil \frac{n+1}{\ell}\right\rceil +r-1.
$$
These two inequalities determine the value $\tc_r[\dot \xi_\ell: \dot E(\xi_\ell) \to \RP^n]$ with indeterminacy 1.
\end{example}

\end{document}